\documentclass[12pt]{amsart}

\usepackage{amsmath}
\usepackage{amsthm}
\usepackage{amsopn}
\usepackage{amssymb}
\usepackage[pdftex]{graphicx}
\usepackage[utf8]{inputenc}
\usepackage[T1]{fontenc}
\usepackage[all]{xy}
\usepackage{mathrsfs}
\usepackage{enumerate}
\usepackage{etoolbox}
\usepackage{textcomp}
\usepackage{mathtools}
\usepackage{bbm}

\newcommand{\ZZ}{\mathbb{Z}}
\newcommand{\CC}{\mathbb{C}}
\newcommand{\RR}{\mathbb{R}}
\newcommand{\TT}{\mathbb{T}}
\newcommand{\FF}{\mathcal{F}}

\newcommand{\SLASH}{\char`\\}
\newcommand{\spacedash}{\hspace{.06cm} - \hspace{.06cm}}
\newcommand{\NN}{\mathbb{N}}
\newcommand{\SSS}{\mathbb{S}}
\newcommand{\defeq}{\vcentcolon=}
\newcommand{\Fin}{\mathcal{F}in}

\DeclareMathOperator*{\hocoeq}{hocoeq}
\DeclareMathOperator*{\colim}{colim}
\DeclareMathOperator{\Res}{Res}
\DeclareMathOperator{\Ind}{Ind}

\newtheorem{thm}[equation]{Theorem}
\newtheorem{lem}[equation]{Lemma}
\newtheorem{prop}[equation]{Proposition}
\newtheorem{cor}[equation]{Corollary}
\theoremstyle{definition} \newtheorem{definition}[equation]{Definition}

\numberwithin{equation}{section}

\newlength\tindent
\renewcommand{\indent}{\hspace*{\tindent}}



\title{Tambara Functors and Commutative Ring Spectra}
\author{John Ullman}

\begin{document}

\begin{abstract}
It is well known that the zeroth stable homotopy group of a genuine equivariant commutative ring spectrum has multiplicative transfers (norms), making it into a Tambara functor. We prove here that all Tambara functors can be obtained in this way. In fact, we prove that the homotopy category of Eilenberg MacLane commutative ring spectra is equivalent to the category of Tambara functors. Several algebraic results on Tambara functors are derived as corollaries. Finally, we rule out the existence of a lax symmetric monoidal construction for Eilenberg MacLane spectra when the group is nontrivial.
\end{abstract}

\maketitle

\section{Introduction and Statements of Results}\label{intro}

Let $G$ be a finite group, and let $X$ be a genuine $G$-equivariant homotopy commutative ring spectrum. Then the Mackey functor $\underline{\pi}_0 X$ is a commutative Green functor. If $X$ is in fact an $E_{\infty}$ ring spectrum, then $\underline{\pi}_0 X$ can be given multiplicative transfer (\emph{norm}) maps, making it into a \emph{Tambara functor} (see~\cite{Brun} or~\cite{Stri}). It is a natural question to ask whether all Tambara functors arise in this way. In this paper, we answer this question in the affirmative.\\
\indent We prove our main theorems in Section~\ref{sec:main}. We paraphrase them here as follows.

\begin{thm}\label{thm:alloccurpara}
[\ref{thm:getanything}] Every Tambara functor occurs as $\underline{\pi}_0$ of a commutative ring spectrum.
\end{thm}

The strategy of proof is a straightforward, generators and relations approach. Let $\underline{R}$ be an arbitrary Tambara functor. We can express $\underline{R}$ as a pushout as below, where we use $\TT ([\spacedash, G/H])$ to denote the free Tambara functor on the represented Mackey functor $[\spacedash, G/H]$.
\begin{align*}
\xymatrix{
 \displaystyle \coprod_{K \subseteq G, y \in ker(f)(G/K)} \textstyle \TT ([\spacedash, G/K]) \ar[d]_-{\coprod_{K,y} 0_*} \ar[rr]^-{\coprod_{K,y} y_*} && \displaystyle \coprod_{H \subseteq G, x \in \underline{R} (G/H)} \textstyle \TT([\spacedash, G/H]) \ar[d]^-{f = \coprod_{H,x} x_*} \\
 \underline{A} \ar[rr] && \underline{R} }
\end{align*}

Here, $\underline{A}$ is the Burnside Mackey functor, which happens to be the initial Tambara functor. It is $\underline{\pi}_0$ of the sphere spectrum. The pushout can be realized by a homotopy pushout of commutative ring spectra; this is Proposition~\ref{prop:commpushout}. The coproducts can be realized by coproducts of commutative ring spectra; this is Proposition~\ref{prop:commcoprod}. The technical heart of this paper is the proof of the following fact: free Tambara functors on represented Mackey functors can be realized by free commutative ring spectra on suspension spectra of finite $G$-sets. This is Theorem~\ref{thm:gentamb}, whose proof occupies all of Section~\ref{sec:commgen}.\\
\indent Next, using these generators and zeroth Postnikov sections $Post^0$, we prove the following.

\begin{thm}\label{thm:homhocommpara}
[\ref{thm:homhocomm}] If $X$ and $H\underline{R}$ are commutative ring spectra and $X$ is $(-1)$-connected, then maps $X \to H\underline{R}$ in the homotopy category of commutative ring spectra correspond bijectively to maps $\underline{\pi}_0 X \to \underline{R}$ of Tambara functors.
\end{thm}

This is precisely analogous to the situation with spectra and Mackey functors. Combining the above two theorems, we obtain the following.

\begin{thm}\label{thm:emcommpara}
[\ref{thm:hoemcomm}] The homotopy category of Eilenberg MacLane commutative ring spectra is equivalent to the category of Tambara functors.
\end{thm}

Several algebraic corollaries are derived from these theorems. We paraphrase the two most important ones here. In the following, we use $\CC$ to denote the derived free commutative ring spectrum functor, and $N_H^G$ to denote the derived norm functor of Hill-Hopkins-Ravenel (\cite{HHR}).

\begin{cor}\label{cor:freetambarapara}
[\ref{cor:existfreetambara}] If $\underline{M}$ is a Mackey functor then $\underline{\pi}_0 \CC (H\underline{M})$ is the free Tambara functor on $\underline{M}$.
\end{cor}

\begin{cor}\label{cor:leftadjointnormpara}
[\ref{cor:leftadjointderivednorm}] Let $H$ be a subgroup of $G$. The left adjoint of restriction from $G$-Tambara functors to $H$-Tambara functors coincides with $N_H^G$ on underlying commutative Green functors.
\end{cor}

In a subsequent paper we shall give detailed algebraic descriptions of the symmetric power and norm constructions on Mackey functors, and shall give an algebraic demonstration of the adjunction in Corollary~\ref{cor:leftadjointnormpara}.\\
\indent We now have the following dictionary for Eilenberg MacLane (EM) spectra.
\begin{align*}
	\text{EM spectra} &\leftrightarrow \text{Mackey functors} \\
	\text{EM (homotopy) rings} &\leftrightarrow \text{Green functors} \\
	\text{EM homotopy commutative rings} &\leftrightarrow \text{Commutative Green functors} \\
	\text{EM commutative rings} &\leftrightarrow \text{Tambara functors}
\end{align*}

The second and third lines use the first line, and the fact that
\begin{align*}
	\underline{\pi}_0 (X \wedge Y) \cong \underline{\pi}_0 X \otimes \underline{\pi}_0 Y
\end{align*}

when $X$ and $Y$ are $(-1)$-connected and $X \wedge Y$ is a derived smash product. The case of associative ring spectra can be handled using a similar approach to the one in this paper, but is vastly easier.\\
\indent In Section~\ref{sec:tambara} we give a definition of Tambara functors and give a direct construction of free Tambara functors on represented Mackey functors. We also give some preliminaries on colimits of Tambara functors. In Section~\ref{sec:commspectra} we prove the aforementioned facts about coproducts and homotopy pushouts of equivariant commutative ring spectra (Propositions~\ref{prop:commcoprod} and~\ref{prop:commpushout}, resp.). Section~\ref{sec:commgen} is devoted to analyzing the free commutative ring spectrum on a finite $G$-set. We state and prove our main theorems in Section~\ref{sec:main}. Finally, in Section~\ref{sec:cgreen} we investigate the extent to which Tambara functors differ from commutative Green functors. In particular, we show that there can be no lax symmetric monoidal construction of Eilenberg MacLane spectra when $G$ is nontrivial, a somewhat disappointing situation.\\
\indent Throughout this paper we use $Mack(G)$ to denote the category of Mackey functors over $G$, and $sMack(G)$ to denote the category of semi-Mackey functors (that is, Mackey functors without additive inverses). We also denote the category of commutative Green functors over $G$ by $Comm(G)$. We index our spectra on a complete $G$-universe.

\section{Tambara and Semi-Tambara Functors}\label{sec:tambara}

We begin by giving a definition of (semi-)Tambara functors, after some preliminaries. Let $G$ be a finite group, and let $\Fin_G$ denote the category of finite $G$-sets. Also let $Set_{\neq \emptyset}$ denote the category of nonempty sets. Let $i : X \to Y$ and $j : Y \to Z$ be maps in $\Fin_G$. Let
\begin{align*}
	\textstyle \prod_{i,j} X \defeq \{ (z, s) : z \in Z, s : j^{-1} (z) \to X, i \circ s = Id \}
\end{align*}

be the set of sections of $i$ defined on fibers of $j$, with $G$ acting by conjugation. There is an obvious $G$-map
\begin{align*}
	p : \textstyle \prod_{i,j} X &\to Z \\
	(z, s) &\mapsto z
\end{align*}

as well as an evaluation $G$-map as below.
\begin{align*}
	e : Y \times_Z \textstyle \prod_{i,j} X &\to X \\
	(y, (z, s)) &\mapsto s(y)
\end{align*}

Observe that the diagram below commutes.
\begin{align*}
\xymatrix{
 Y \times_Z \textstyle \prod_{i,j} X \ar[d]_-{e} \ar[rr]^-{\pi_2} & & \textstyle \prod_{i,j} X \ar[d]^-{p} \\
 X \ar[r]_-{i} & Y \ar[r]_-{j} & Z }
\end{align*}

An \emph{exponential diagram} is any diagram in $\Fin_G$ which is isomorphic to one of the above form. If the diagram
\begin{align*}
\xymatrix{
 A \ar[d]_-{f} \ar[rr]^-{g} & & B \ar[d]^-{h} \\
 X \ar[r]_-{i} & Y \ar[r]_-{j} & Z }
\end{align*}

is an exponential diagram we will say that $(f,g,h)$ is a \emph{distributor} for $(i,j)$. We can now define semi-Tambara functors. Our definition is equivalent to that of Tambara's "semi-TNR functors" in~\cite{Tambara}.

\begin{definition}\label{def:tambara}
A \emph{semi-Tambara functor} $\underline{M}$ is a triplet of functors
\begin{align*}
	\underline{M}^* : \Fin_G^{op} &\to Set_{\neq \emptyset} \\
	\underline{M}_* : \Fin_G &\to Set_{\neq \emptyset} \\
	\underline{M}_{\star} : \Fin_G &\to Set_{\neq \emptyset}
\end{align*}
with common object assignment $X \mapsto \underline{M} (X)$ such that
\begin{enumerate}[(i)]
\item if $X \xrightarrow{i} Z \xleftarrow{j} Y$ is a coproduct in $\Fin_G$ then
\begin{align*}
	\underline{M} (X) \xleftarrow{\underline{M}^* (i)} \underline{M} (Z) \xrightarrow{\underline{M}^* (j)} \underline{M} (Y)
\end{align*}
is a product in $Set_{\neq \emptyset}$,
\item for any pullback diagram
\begin{align*}
\xymatrix{
 P \ar[d]_-{p} \ar[r]^-{q} & Y \ar[d]^-{i} \\
 X \ar[r]_-{j} & Z }
\end{align*}
we have the two relations $\underline{M}^* (j) \circ \underline{M}_* (i) = \underline{M}_* (p) \circ \underline{M}^* (q)$ and $\underline{M}^* (j) \circ \underline{M}_{\star} (i) = \underline{M}_{\star} (p) \circ \underline{M}^* (q)$, and
\item for any exponential diagram
\begin{align*}
\xymatrix{
 A \ar[d]_-{f} \ar[rr]^-{g} & & B \ar[d]^-{h} \\
 X \ar[r]_-{i} & Y \ar[r]_-{j} & Z }
\end{align*}
we have $\underline{M}_{\star} (j) \circ \underline{M}_* (i) = \underline{M}_* (h) \circ \underline{M}_{\star} (g) \circ \underline{M}^* f$.
\end{enumerate}
A map of semi-Tambara functors $\underline{M} \to \underline{N}$ is a collection of maps of sets $\underline{M} (X) \to \underline{N} (X)$ which forms a triplet of natural transformations $\underline{M}^* \to \underline{N}^*$, $\underline{M}_* \to \underline{N}_*$, $\underline{M}_{\star} \to \underline{N}_{\star}$. We denote the category of semi-Tambara functors by $sTamb(G)$.
\end{definition}

The third condition above is called the \emph{distributive law}. If $\underline{M}$ is a semi-Tambara functor the structure maps $\underline{M}^* (f)$ are called \emph{restrictions}, the $\underline{M}_* (f)$ are called \emph{transfers} and the $\underline{M}_{\star} (f)$ are called \emph{norms}. When the choice of $\underline{M}$ is clear we will denote these by $r_f$, $t_f$ and $n_f$, respectively. \\
\indent Now for any $X \in \Fin_G$, the composite
\begin{align*}
	\underline{M} (X) \times \underline{M} (X) \cong \underline{M} (X \textstyle \coprod X) \xrightarrow{\underline{M}_* (Id_X \coprod Id_X)} \underline{M} (X)
\end{align*}

defines an operation making $\underline{M}(X)$ into a commutative monoid. We call this addition; the unit (zero) comes from the unique transfer
\begin{align*}
	\underline{M}_* : \underline{M} (\emptyset) \to \underline{M} (X).
\end{align*}

A \emph{Tambara functor} is a semi-Tambara functor $\underline{M}$ such that these monoids are abelian groups. We denote by $Tamb(G)$ the category of Tambara functors. Note that we obtain analogous definitions of semi-Mackey and Mackey functors by deleting the norms from the above definition.\\
\indent Next, using norm maps instead, we obtain a second operation which distributes over the first. We call this multiplication. With these commutative semi-ring structures the restrictions become maps of rings, the transfers are maps of modules, and the norms are maps of multiplicative monoids. Thus one obtains a "forgetful" functor from $Tamb(G)$ to $Comm(G)$. We also obtain forgetful functors $sTamb(G) \to sMack(G)$ and $Tamb(G) \to Mack(G)$ by neglect of the norms.\\
\indent Recall that (semi-)Mackey functors can be defined equivalently in terms of the $\underline{M} (G/H)$, with structure maps corresponding to isomorphisms of orbits (\emph{conjugations}) and restrictions $r_K^H$ and transfers $t_K^H$ associated to the canonical projections $G/K \to G/H$ for $K \subseteq H$. Similarly, it is possible to define (semi-)Tambara functors in terms of these sets and maps, the sum and product operations and the corresponding norm maps $n_K^H$. The resulting description is rather cumbersome; however, it is at least easy to see that this data determines the Tambara functor, and that maps of Mackey functors respecting this data are maps of Tambara functors. Also, assuming $K \subsetneq H$ we have (schematically)
\begin{itemize}
\item $n_K^H (a + b) = n_K^H (a) + n_K^H (b) + t( ... )$, and
\item $n_K^H t_L^K = t( ... )$ when $L \subsetneq K$,
\end{itemize}

where the "t( ... )" are sums of transfers of expressions involving subgroups that are smaller than $H$. This allows us to do inductive proofs. For example, if $\underline{M}$ and $\underline{N}$ are Tambara functors and $f : \underline{M} \to \underline{N}$ is a map of commutative Green functors, and if each $\underline{M} (G/K)$ is generated as a ring by transfers of elements up from proper subgroups and elements on which $f$ commutes with the norm maps $n_K^H$, then $f$ is a map of Tambara functors.\\
\indent The Grothendieck group construction gives left adjoints
\begin{align*}
	sMack(G) &\to Mack(G)\\
	sTamb(G) &\to Tamb(G)
\end{align*}

to the appropriate forgetful functors. For Mackey functors this is trivial; for Tambara functors see~\cite{Tambara} (or, alternatively, Section~13 of~\cite{Stri}). Next we give a definition of free Tambara functors.
\begin{definition}\label{def:freetambara}
Let $\underline{M}$ be a Mackey functor. A \emph{free Tambara functor} on $\underline{M}$ is a Tambara functor $\TT(\underline{M})$ together with a map of Mackey functors $\underline{M} \to \TT(\underline{M})$ which is initial among maps from $\underline{M}$ to Tambara functors. A \emph{free semi-Tambara functor} on a semi-Mackey functor $\underline{M}$ is a semi-Tambara functor $s\TT(\underline{M})$ together with a map of semi-Mackey functors $\underline{M} \to s\TT(\underline{M})$ which is initial among maps from $\underline{M}$ to semi-Tambara functors.
\end{definition}

\indent \emph{Remark:} It is not difficult to give an algebraic proof that the forgetful functor from (semi-)Tambara functors to (semi-)Mackey functors has a left adjoint, and therefore that every Mackey functor generates a free Tambara functor. However, we do not need to assume this, and in fact will obtain the existence of free Tambara functors as a corollary.\\
\indent We now give a description of the free Tambara functor on a represented Mackey functor. Fix a finite $G$-set $T$.\\
\indent Let $X$ be a finite $G$-set. Let $F_T (X)$ be the set of isomorphism classes of diagrams in $\Fin_G$ of the form below.
\begin{align}\label{eq:freetambdiag}
\xymatrix{
 & U \ar[dl] \ar[r] & V \ar[dr] & \\
 T & & & X }
\end{align}

Such a diagram represents a transfer of a norm of a restriction of the universal element. Now if $X \to Y$ is a map of finite $G$-sets, we define the corresponding transfer of the above diagram by composing $V \to X$ with $X \to Y$. If instead $Y \to X$, we define the corresponding restriction of~\ref{eq:freetambdiag} by the sequence of arrows on the bottom of the diagram below, where the squares are pullbacks and the triangle commutes.
\begin{align*}
\xymatrix{
 T & U \ar[l] \ar[r] & V \ar[r] & X \\
    & Q \ar[ul] \ar[u] \ar[r] & P \ar[u] \ar[r] & Y \ar[u] }
\end{align*}

Addition is then given by taking disjoint unions of the $U$'s and $V$'s; the diagram with $U$ and $V$ empty is the additive unit. To define norm maps, we take our cue from the distributive law. Given a map $X \to Y$ of finite $G$-sets, we define the corresponding norm of~\ref{eq:freetambdiag} by the sequence of arrows on the bottom in the diagram below,
\begin{align*}
\xymatrix{
 & U \ar[dl] \ar[r] & V \ar[dr] & \\
 T & P \ar[l] \ar[r] \ar[dr] \ar[u] & A \ar[u] \ar[d] & X \ar[d] \\
 & & B \ar[r] & Y }
\end{align*}

where the square is a pullback, the triangles commute, and the trapezoid is exponential. To show that $F_T$ is a semi-Tambara functor, we use the following three lemmas, which we state without proof.

\begin{lem}\label{lem:distrpullback}
(Commutation of norms and restrictions) Suppose given maps of finite $G$-sets as below, where the squares are pullbacks.
\begin{align*}
\xymatrix{
 X \ar[r]^-{i} & Y \ar[r]^-{j} & Z \\
 Q \ar[u] \ar[r]_-{g} & P \ar[u] \ar[r]_-{h} & W \ar[u]_-{k} }
\end{align*}
Then the pullback over $k$ of the exponential diagram for $i,j$ is the exponential diagram for $g,h$.
\end{lem}

\begin{lem}\label{lem:distrlaw}
(Distributive law) Suppose given a commutative diagram of finite $G$-sets as below.
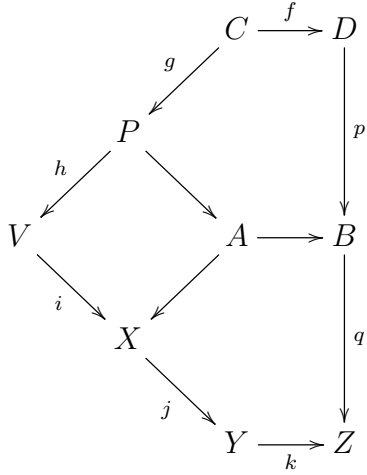
\begin{align*}
\xymatrix{
 & & C \ar[dl]_-{g} \ar[r]^-{f} & D \ar[dd]^-{p} \\
 & P \ar[dl]_-{h} \ar[dr] & & \\
 V \ar[dr]_-{i} & & A \ar[dl] \ar[r] & B \ar[dd]^-{q} \\
 & X \ar[dr]_-{j} & & \\
 & & Y \ar[r]_-{k} & Z }
\end{align*}
If the square is a pullback and the two interior pentagons are exponential, then the outer pentagon is exponential; that is, the maps $hg$, $f$ and $qp$ form a distributor for $ji, k$.
\end{lem}

\begin{lem}\label{lem:functorialnorm}
(Functorality of norm) Suppose given a commutative diagram of finite $G$-sets as below.
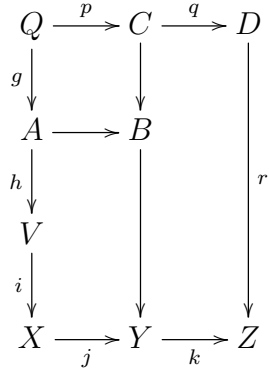
\begin{align*}
\xymatrix{
 Q \ar[d]_-{g} \ar[r]^-{p} & C \ar[d] \ar[r]^-{q} & D \ar[ddd]^-{r} \\
 A \ar[d]_-{h} \ar[r] & B \ar[dd] & \\
 V \ar[d]_-{i} & & \\
 X \ar[r]_-{j} & Y \ar[r]_-{k} & Z }
\end{align*}
If the square is a pullback and the two interior rectangles are exponential, then the outer rectangle is exponential; that is, the maps $hg$, $qp$ and $r$ form a distributor for $i,kj$.
\end{lem}

Now let $\theta_T \in F_T (T)$ be the element represented by the diagram below.
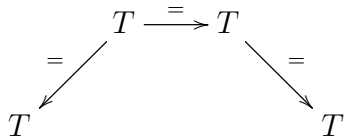
\begin{align*}
\xymatrix{
 & T \ar[dl]_-{=} \ar[r]^-{=} & T \ar[dr]^-{=} & \\
 T & & & T }
\end{align*}

\begin{prop}\label{prop:repsemitamb}
If $\underline{R}$ is a (semi-)Tambara functor then the map shown below is an isomorphism.
\begin{align*}
	Hom_{sTamb(G)} (F_T, \underline{R}) &\to \underline{R}(T)\\
	                                                      f      &\mapsto f(\theta_T)
\end{align*}
That is, $F_T$ is the free semi-Tambara functor on (i.e. $s\TT$ of) the semi-Mackey functor represented by $T$.
\end{prop}
\begin{proof}
It is readily verified that the diagram
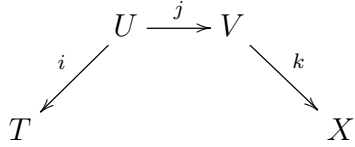
\begin{align*}
\xymatrix{
 & U \ar[dl]_-{i} \ar[r]^-{j} & V \ar[dr]^-{k} & \\
 T & & & X }
\end{align*}
is equal to $t_k n_j r_i (\theta_T)$. It follows immediately that the map in the statement is injective. Next, given an element $x \in \underline{R}(T)$, we define a map $F_T \to \underline{R}$ by sending the above diagram to $t_k n_j r_i (x)$. It is straightforward to show that this is a well-defined map of semi-Tambara functors, and it clearly maps $\theta_T$ to $x$.
\end{proof}

Now denote by $F_T^+$ the levelwise additive completion of $F_T$. The following is immediate.

\begin{cor}\label{cor:reptamb}
$F_T^+$ is the free Tambara functor on the represented Mackey functor $[\spacedash, T]$, i.e. $F_T^+ = \TT ([\spacedash, T])$.
\end{cor}

We now examine the structure of $F_T (X)$ for a given finite $G$-set $X$. Decomposing $V$ into orbits, we see that $F_T (X)$ is freely generated as a commutative monoid by the diagrams where $V$ is an orbit; these diagrams therefore provide a canonical $\ZZ$-basis for $F_T^+ (X)$. Next, for each $n \geq 0$, we define $F_T [n] (X)$ to be the subset of $F_T (X)$ consisting of those diagrams
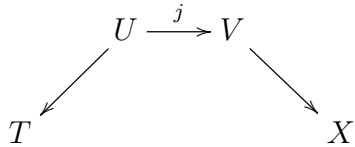
\begin{align*}
\xymatrix{
 & U \ar[dl] \ar[r]^-{j} & V \ar[dr] & \\
 T & & & X }
\end{align*}

such that $j^{-1} (v)$ has $n$ elements for all $v \in V$; these clearly form a semi-Mackey functor, and $F_T [n] (X)$ is freely generated by the diagrams in $F_T [n] (X)$ such that $V$ is an orbit. Letting $F_T^+ [n] (X)$ denote the additive completion of $F_T [n] (X)$, we obtain direct sum decompositions as below.
\begin{align*}
	F_T &= \bigoplus_{n \in \NN} F_T [n] \\
	F_T^+ &= \bigoplus_{n \in \NN} F_T^+ [n]
\end{align*}

Note that each $F_T^+ [n] (X)$ is a \emph{finitely generated} free abelian group, since the basis diagrams consist of sets with bounded cardinality. Next, recall that the Burnside Mackey functor $\underline{A}$ is the initial Tambara functor. Since the multiplicative unit in $F_T^+ (X)$ is the diagram below,
\begin{align*}
\xymatrix{
 & \emptyset \ar[dl] \ar[r] & X \ar[dr]^-{=} & \\
 T & & & X }
\end{align*}

it follows that the unique map of Tambara functors $\underline{A} \to F_T^+$ maps the span $X \xleftarrow{f} V \to \ast$, which is $t_f$ of $1 \in \underline{A} (V)$, to the following diagram.
\begin{align*}
\xymatrix{
 & \emptyset \ar[dl] \ar[r] & V \ar[dr]^-{f} & \\
 T & & & X }
\end{align*}

Thus we obtain an induced isomorphism
\begin{align}\label{eq:ft0grad}
	\underline{A} \xrightarrow{\cong} F_T^+ [0].
\end{align}

Meanwhile, the canonical map of Mackey functors $[\spacedash, T] \to F_T^+$ maps the span $X \xleftarrow{i} V \xrightarrow{j} T$ to the diagram below,
\begin{align*}
\xymatrix{
 & V \ar[dl]_-{j} \ar[r]^-{=} & V \ar[dr]^-{i} & \\
 T & & & X }
\end{align*}

so we obtain an induced isomorphism
\begin{align}\label{eq:ft1grad}
	[\spacedash, T] \xrightarrow{\cong} F_T^+ [1].
\end{align}

\indent Next we consider restricting Tambara functors to subgroups. Let $H$ be a subroup of $G$. If $\underline{R}$ is a $G$-Tambara functor its restriction to $H$ is naturally an $H$-Tambara functor; we simply take
\begin{align*}
	(\Res_H^G \underline{R})_{\star} \defeq \underline{R}_{\star} \circ (G \times_H (\spacedash)),
\end{align*}
just as with the restrictions and transfers. Since the induction functor $G \times_H (\spacedash) : \Fin_H \to \Fin_G$ preserves colimits, pullback diagrams and exponential diagrams, we see that we have a restriction functor
\begin{align*}
	\Res_H^G : Tamb(G) \to Tamb(H).
\end{align*}

Now, any map of $G$-sets $f : X \to G \times_H Y$ identifies $X$ canonically as $G \times_H f^{-1} Y$, so diagrams of $G$-sets
\begin{align*}
	T \leftarrow U \to V \to G \times_H X
\end{align*}

are in bijection with diagrams of $H$-sets
\begin{align*}
	\Res_H^G T \leftarrow U' \to V' \to X,
\end{align*}

and thus we see that
\begin{align*}
	\Res_H^G (F_T^+) \cong F_{\Res_H^G T}^+.
\end{align*}

We leave the following simple verification to the reader.

\begin{prop}\label{prop:resft}
The unique map $F_{\Res_H^G T}^+ \to \Res_H^G (F_T^+)$ of Tambara functors making the following diagram commute is an isomorphism.
\begin{align*}
\xymatrix{
 [\spacedash, \Res_H^G T] \ar[d]_-{\cong} \ar[rr]^-{\theta_{\Res_H^G T}} & & F_{\Res_H^G T}^+ \ar@{.>}[d] \\
 \Res_H^G [\spacedash, T] \ar[rr]_-{\Res_H^G (\theta_T)} & & \Res_H^G (F_T^+) }
\end{align*}
\end{prop}

We will need a few easy facts about colimits in $Tamb(G)$.

\begin{lem}\label{lem:dirlimtamb}
The category $Tamb(G)$ has all direct limits. They are computed levelwise as direct limits of sets.
\end{lem}

\begin{lem}\label{lem:weaktambcoprod}
Let $\underline{R}_1$ and $\underline{R}_2$ be Tambara functors. If $\underline{R}_1 \otimes \underline{R}_2$ has a Tambara functor structure (which induces the canonical product) such that the inclusions $j_i : \underline{R}_i \to \underline{R}_1 \otimes \underline{R}_2$ are maps of Tambara functors then it is the coproduct of $\underline{R}_1$ and $\underline{R}_2$ in $Tamb(G)$.
\end{lem}
\begin{proof}
Suppose we have maps $f_i : \underline{R}_i \to \underline{R}_3$ of Tambara functors for $i=1,2$. Then there is a unique map of commutative Green functors $f : \underline{R}_1 \otimes \underline{R}_2 \to \underline{R}_3$ such that $f j_1 = f_1$ and $f j_2 = f_2$. Thus it suffices to show that $f$ commutes with the norm maps $n_K^H$ for all pairs of subgroups $K \subseteq H$ of $G$. Since we know that it does on simple tensors $x \otimes y$, and that every element in $\underline{R}_1 \otimes \underline{R}_2 (G/K)$ is a sum of transfers of such elements, this can be proven by induction on the order of $H$ using the distributive law.
\end{proof}

\begin{cor}\label{cor:weaktambpushout}
Let $\underline{R}_1$ and $\underline{R}_2$ be Tambara functors as in Lemma~\ref{lem:weaktambcoprod}, and let $g_i : \underline{R}_0 \to \underline{R}_i$ be maps of Tambara functors for $i=1,2$. If $\underline{R}_1 \otimes_{\underline{R}_0} \underline{R}_2$ has a Tambara functor structure such that the quotient map
\begin{align*}
	\underline{R}_1 \otimes \underline{R}_2 \xrightarrow{q} \underline{R}_1 \otimes_{\underline{R}_0} \underline{R}_2
\end{align*}
is a map of Tambara functors then it is the pushout in $Tamb(G)$.
\end{cor}
\begin{proof}
Suppose we have maps $f_i : \underline{R}_i \to \underline{R}_3$ in $Tamb(G)$ for $i=1,2$ such that $f_1 g_1 = f_2 g_2$. Then there is a unique map of commutative Green functors $f : \underline{R}_1 \otimes_{\underline{R}_0} \underline{R}_2 \to \underline{R}_3$ such that $fqj_1 = f_1$ and $fqj_2 = f_2$. By Lemma~\ref{lem:weaktambcoprod}, the map $fq$ is a map of Tambara functors, so $f$ is as well because $q$ is surjective.
\end{proof}

\section{Commutative Ring $G$-Spectra}\label{sec:commspectra}

In this section we analyze the effects of coproducts and pushouts on $\underline{\pi}_0$ of commutative ring spectra, which will eventually allow us to construct our "free resolutions" of Tambara functors. We work in the category of orthogonal $G$-spectra, which we denote by $Sp_G$, since this is the most convenient context for analyzing norm constructions. We utilize the $\SSS$ model structures of~\cite{Stolz}, so that induction functors are left Quillen functors. Denote by $comm_G$ the category of commutative ring orthogonal $G$-spectra, and recall that we can pull back the positive stable $\SSS$ model structure to this category. We refer to~\cite{Stolz} and Section~A.4 of~\cite{Ull} for background on these model structures, which build on the classical ones from~\cite{MM}. We shall utilize the following definitions.

\begin{definition}\label{def:flat}
If $k$ is a commutative ring spectrum and $X$ is a $k$-module, we say that $X$ is \emph{flat} if the functor $X \wedge_k (\spacedash)$ preserves weak equivalences between $k$-modules. An \emph{h-cofibration} of $k$-modules is a map $X \to Y$ of $k$-modules that has the homotopy extension property.
\end{definition}

If no coefficient ring $k$ is mentioned, we will take $k = S$, the sphere spectrum. Note that h-cofibrations are spacewise closed inclusions.

\begin{lem}\label{lem:flatness}
The following classes of spectra are flat.
\begin{enumerate}[(i)]
\item Cofibrant spectra
\item Symmetric powers of positive cofibrant spectra
\item Cofibrant commutative ring spectra
\end{enumerate}
\end{lem}

For proofs, see Section~A.4 of~\cite{Ull}.

\begin{cor}\label{cor:moduleflatness}
If $k$ is a commutative ring spectrum then cofibrant $k$-modules are flat (as $k$-modules).
\end{cor}

Next we recall the free commutative ring spectrum functor. It is given by the following.

\begin{align*}
	\CC(X) \defeq S \vee X \vee X^{\wedge 2}/\Sigma_2 \vee X^{\wedge 3}/\Sigma_3 \vee ...
\end{align*}

Now suppose that $i : A \to B$ is a generating positive cofibration of spectra. For any $j > 0$, denote by $\partial_A B^{\wedge j}$ the "union of the images" of the maps $B^{\wedge j-k-1} \wedge A \wedge B^{\wedge k} \to B^{\wedge j}$. One now has a filtration $\{Y_j\}$ of $\CC (B)$ by $\CC (A)$-submodules such that $Y_0 = \CC (A)$ and there are pushout diagrams of $\CC (A)$-modules as below.
\begin{align*}
\xymatrix{
 \CC (A) \wedge (\partial_A B^{\wedge j})/\Sigma_j \ar[d] \ar[r] & \CC (A) \wedge B^{\wedge j}/\Sigma_j \ar[d] \\
 Y_{j-1} \ar[r] & Y_j }
\end{align*}

Now suppose that we have a pushout diagram as below.
\begin{align*}
\xymatrix{
 \CC (A) \ar[d] \ar[r]^-{\CC (i)} & \CC (B) \ar[d] \\
 X \ar[r] & Y }
\end{align*}

Applying the functor $X \wedge_{\CC (A)} (\spacedash)$ to the above filtration on $\CC (B)$, we obtain a filtration $\{Y_j\}$ of $Y$ by $X$-modules such that $Y_0 = X$ and for any $j > 0$ the map $Y_{j-1} \to Y_j$ is an h-cofibration with quotient isomorphic to $X \wedge (B/A)^{\wedge j}/\Sigma_j$. We can now prove the following.

\begin{lem}\label{lem:commcofflat}
Suppose we have a pushout diagram
\begin{align*}
\xymatrix{
 X \ar[d] \ar[r] & Z \ar[d] \\
 Y \ar[r] & P }
\end{align*}
in $comm_G$. If $X \to Y$ is a cofibration then $P = Z \wedge_X Y$ is a derived smash product in the category of $X$-modules.
\end{lem}
\begin{proof}
By Corollary~\ref{cor:moduleflatness} it suffices to show that $Y$ is flat as an $X$-module. Now we may assume that $X \to Y$ is a cell in the generating cofibrations, since it is a retract of one, and flat modules are closed under retract. By attaching one cell at a time, we obtain a transfinite filtration $\{ X_{\alpha} \}$ of $Y$ such that
\begin{itemize}
\item $X_0 = X$,
\item $X_{\alpha} = \varinjlim_{\gamma < \alpha} X_{\gamma}$ when $\alpha$ is a limit element, and
\item for each $\alpha$ with a succesor there is a generating positive cofibration $i_{\alpha} : A_{\alpha} \to B_{\alpha}$ and a pushout diagram in $comm_G$ as below.
\begin{align*}
\xymatrix{
 \CC (A_{\alpha}) \ar[d] \ar[r]^-{\CC (i_{\alpha})} & \CC (B_{\alpha}) \ar[d] \\
 X_{\alpha} \ar[r] & X_{\alpha+1} }
\end{align*}
\end{itemize}
By the above analysis of such pushouts, they are h-cofibrations of $X_{\alpha}$-modules; hence, h-cofibrations of $X$-modules by neglect of structure. Now $X$ is certainly a flat $X$-module, so by transfinite induction it suffices to show that $X_{\alpha+1}$ is a flat $X$-module if $X_{\alpha}$ is. Of course, $X_{\alpha+1}$ is filtered by h-cofibrations of $X$-modules with successive quotients of the form $X_{\alpha} \wedge (B_{\alpha}/A_{\alpha})^j/\Sigma_j$, so the result follows by Lemma~\ref{lem:flatness}(ii) and the inductive hypothesis.
\end{proof}

Next we recall that the zeroth homotopy group of a commutative ring spectrum is (naturally) a Tambara functor; that is, we have a functor
\begin{align*}
	\underline{\pi}_0 : comm_G \to Tamb(G).
\end{align*}

For proof, see~\cite{Brun} or~\cite{Stri}. We have the following facts.

\begin{prop}\label{prop:commcoprod}
If $\{ X_{\alpha} \}$ is a collection of cofibrant, $(-1)$-connected commutative ring spectra then $\underline{\pi}_0 (\coprod_{\alpha} X_{\alpha})$ is the coproduct of the $\underline{\pi}_0 (X_{\alpha})$ in $Tamb(G)$.
\end{prop}
\begin{proof}
The case of two factors follows from Lemmas~\ref{lem:flatness}(iii) and~\ref{lem:weaktambcoprod}. The case of a finite collection then follows by induction. The case of an arbitrary collection follows by considering the direct limit over finite subcollections, using Lemma~\ref{lem:dirlimtamb} and the fact that all the structure maps are cofibrations in $comm_G$, hence, h-cofibrations. 
\end{proof}

\begin{prop}\label{prop:commpushout}
Suppose we have a pushout diagram
\begin{align*}
\xymatrix{
 X \ar[d] \ar[r] & Z \ar[d] \\
 Y \ar[r] & P }
\end{align*}
in $comm_G$. If $X \to Y$ is a cofibration, with $X$, $Y$ and $Z$ all cofibrant and $(-1)$-connected, then so is $P$, $\underline{\pi}_0 (P) \cong \underline{\pi}_0 (Z) \otimes_{\underline{\pi}_0 (X)} \underline{\pi}_0 (Y)$ and
\begin{align*}
\xymatrix{
 \underline{\pi}_0 (X) \ar[d] \ar[r] & \underline{\pi}_0 (Z) \ar[d] \\
 \underline{\pi}_0 (Y) \ar[r] & \underline{\pi}_0 (P) }
\end{align*}
is a pushout diagram in $Tamb(G)$.
\end{prop}

This follows directly from Lemmas~\ref{lem:flatness}(iii) and~\ref{lem:commcofflat} and Corollary~\ref{cor:weaktambpushout}. We will need the following fact about Postnikov sections.

\begin{prop}\label{prop:commpost0}
If $X$ is a $(-1)$-connected commutative ring spectrum, then there is an initial map
\begin{align*}
	X \to Post^0 X
\end{align*}
in $Ho(comm_G)$ from $X$ to a $1$-coconnected commutative ring spectrum. This map induces an isomorphism on $\underline{\pi}_0$, and $Post^0 X$ is $(-1)$-connected. This map can even be constructed functorially on the point-set level so that the map is a cofibration and the target is fibrant.
\end{prop}

To construct $Post^0 X$, one simply kills all elements of the higher homotopy groups infinitely many times by attaching null-homotopies, pausing at each stage to take a fibrant replacement. Since $comm_G$ is compactly generated, each step can be done functorially (as long as one kills \emph{all maps} representing elements of the homotopy groups rather than just one from each homotopy class).\\
\indent We need one more ingredient before we can build a commutative ring spectrum with an arbitrary $\underline{\pi}_0$.

\section{The Free Commutative Ring Spectrum on a $G$-Set}\label{sec:commgen}

In this section we show that we can realize the free Tambara functor on a represented Mackey functor as $\underline{\pi}_0$ of a commutative ring spectrum. Fix a finite $G$-set $T$, regard it as a discrete $G$-space and consider the free commutative ring spectrum on $T$, as below. Here, $F_1 S^1$ denotes the free orthogonal spectrum on $S^1$ in level $\RR^1$, which is a positive cofibrant replacement for the sphere spectrum.
\begin{align*}
	\CC (F_1 S^1 \wedge T_+) = S \vee (F_1 S^1 \wedge T_+) \vee (F_1 S^1 \wedge T_+)^{\wedge 2}/\Sigma_2 \vee ...
\end{align*}

Let $\CC_T \defeq \underline{\pi}_0 \CC (F_1 S^1 \wedge T_+)$. The inclusion of the wedge summand $F_1 S^1 \wedge T_+$ and the standard weak equivalence $F_1 S^1 \wedge T_+ \xrightarrow{\sim} \Sigma^{\infty} T_+$ induce a canonical map
\begin{align}\label{eq:geninclude}
	\iota_T : [ \spacedash, T] \to \CC_T
\end{align}

of Mackey functors, and this in turn induces a unique map of Tambara functors
\begin{align}\label{eq:gentambiso}
	\psi_T : F_T^+ \to \CC_T.
\end{align}

We shall require the following theorem; the proof will occupy the rest of this section.

\begin{thm}\label{thm:gentamb}
The map $\psi_T$ is an isomorphism of Tambara functors. That is, $\underline{\pi}_0$ of the free commutative ring spectrum on a finite $G$-set is the free Tambara functor on the Mackey functor represented by that $G$-set.
\end{thm}

First observe that the wedge sum decomposition above induces a direct sum decomposition of Mackey functors:
\begin{align*}
	\CC_T = \bigoplus_{n \in \NN} \CC_T [n]
\end{align*}

with $\CC_T [0] \cong \underline{A}$, $\CC_T [1] \cong [\spacedash, T]$ and
\begin{align*}
	\CC_T [n] = \underline{\pi}_0 \big((F_1 S^1 \wedge T_+)^{\wedge n}/\Sigma_n\big)
\end{align*}

for $n \geq 2$. In view of~\ref{eq:ft0grad} and~\ref{eq:ft1grad}, the lemma below follows by considering the definition of the norm maps for a commutative ring spectrum.

\begin{lem}\label{lem:01comp}
The map $\psi_T$ is a graded map of Mackey functors. The components $\psi_T[0]$ and $\psi_T[1]$ are isomorphisms.
\end{lem}

We now proceed by induction on the order of $G$. For the base case, we have the following.

\begin{lem}\label{lem:trivgroup}
The map $\psi_T$ is an isomorphism when $G$ is trivial.
\end{lem}
\begin{proof}
Mackey functors over the trivial group are just abelian groups, and Tambara functors are just commutative rings. Lemma~15.5 of \cite{MMSS} implies that the map
\begin{align*}
	{E\Sigma_n}_+ \wedge_{\Sigma_n} (F_1 S^1 \wedge T_+)^{\wedge n} \to (F_1 S^1 \wedge T_+)^{\wedge n}/\Sigma_n
\end{align*}
is a weak equivalence for all $n > 1$. It follows from this that we have $\pi_0 ((F_1 S^1 \wedge T_+)^{\wedge n}/\Sigma_n) \cong \ZZ \{ T \}^{\otimes n}/\Sigma_n$, and therefore $\CC_T$ is the free commutative ring on $T$.
\end{proof}

Proposition~\ref{prop:resft} implies the following.

\begin{lem}\label{commrest}
Let $H$ be a subgroup of $G$. If $X \in comm_G$ then
\begin{align*}
	\underline{\pi}_0 (\Res_H^G X) = \Res_H^G (\underline{\pi}_0 X)
\end{align*}
as a Tambara functor, so that
\begin{align*}
	\Res_H^G \psi_T \cong \psi_{\Res_H^G T}.
\end{align*}
\end{lem}

Thus, we may assume inductively that $\psi_T$ is an isomorphism on $G/H$ for all proper subgroups $H$ of $G$. Now let $n > 1$, and denote by $\FF_G[n]$ the family of subgroups of $G \times \Sigma_n$ that have trivial intersection with $1 \times \Sigma_n$. For each subgroup $H \subseteq G$ and homomorphism $\phi : H \to \Sigma_n$ the set $H^{\phi} \defeq \{(h,\phi(h)) : h \in H\}$ is in $\FF_G[n]$, and every element of $\FF_G[n]$ is of this form for unique $H$ and $\phi$. We will require some lemmas about these subgroups.

\begin{lem}\label{lem:freesigmaorbit}
Let $H^{\phi}$ be a subgroup of $G \times \Sigma_n$ in $\FF_G[n]$, and let $X$ be a $(G \times \Sigma_n)$-spectrum (or space). Then there is a natural isomorphism
\begin{align*}
	(G \times \Sigma_n / H^{\phi})_+ \wedge_{\Sigma_n} X \hspace{.5cm} \cong \hspace{.5cm} G_+ \wedge_H X^{\phi},
\end{align*}
where $X^{\phi}$ is $X$ with $H$-action multiplied by the pullback of the $\Sigma_n$-action along $\phi$.
\end{lem}
\begin{proof}
The above spectrum is
\begin{align*}
	((G \times \Sigma_n / H^{\phi})_+ \wedge X)/\Sigma_n &\cong ((G \times \Sigma_n)_+ \wedge_{H^{\phi}} X)/\Sigma_n \\
	&\cong (\Sigma_n \SLASH G \times \Sigma_n)_+ \wedge_{H^{\phi}} X \cong G_+ \wedge_{H^{\phi}} X,
\end{align*}
where $H^{\phi}$ acts on $G$ via its projection onto $H$. The last spectrum above can be described equivalently as $G_+ \wedge_H X^{\phi}$.
\end{proof}

\begin{lem}\label{lem:freesigmanorm}
Let $n > 1$ and let $\phi : H \to \Sigma_n$ be a homomorphism, so that $H$ acts on $\{1,...,n\}$. Let $t_1, ... , t_m$ be a set of orbit representatives for this $H$-set. Then for $X \in Sp_H$ there is a natural isomorphism
\begin{align*}
	(X^{\wedge n})^{\phi} \cong \bigwedge_{j = 1}^{m} N_{K_j}^H (\Res_{K_j}^H X),
\end{align*}
where $K_j$ is the stabilizer of $t_j$ and $N_{K_j}^H : Sp_{K_j} \to Sp_H$ is the norm functor of~\cite{HHR}.
\end{lem}
\begin{proof}
Splitting the factors into orbits, we clearly have
\begin{align*}
	(X^{\wedge n})^{\phi} \cong \bigwedge_{j = 1}^{m} \big( \bigwedge_{hK_j \in H/K_j} X \big).
\end{align*}
Choosing sets of coset representatives $\{ h_{ij} K_j \}$ for each $j$, we have isomorphisms as below.
\begin{align*}
	\wedge_{h_{ij}K_j} h_{ij} \cdot : N_{K_j}^H (\Res_{K_j}^H X) \xrightarrow{\cong} \bigwedge_{h_{ij}K_j} X
\end{align*}
\end{proof}

Since norm functors preserve weak equivalences between cofibrant spectra and preserve cofibrancy (see Section~I.5 of~\cite{Ull}), as do restriction and induction functors, we obtain the following.

\begin{cor}\label{cor:extpowequiv}
Let $n > 1$ and let $H^{\phi} \in \FF_G[n]$. Then the functor $(G \times \Sigma_n / H^{\phi})_+ \wedge_{\Sigma_n} (\spacedash)^{\wedge n}$ preserves weak equivalences between cofibrant spectra.
\end{cor}

Now denote the universal $\FF_G[n]$-space by $E_G \Sigma_n$. Since we know from Lemma~III.8.4 of~\cite{MM} that the map
\begin{align*}
	{E_G \Sigma_n}_+ \wedge_{\Sigma_n} (F_1 S^1 \wedge T_+)^{\wedge n} \to (F_1 S^1 \wedge T_+)^{\wedge n}/\Sigma_n
\end{align*}

is a weak equivalence, we see from this corollary and Lemma~\ref{lem:freesigmaorbit} that the cells of ${E_G \Sigma_n}$ give us a "cell" structure for the symmetric power. That is, ${E_G \Sigma_n}_+ \wedge_{\Sigma_n} (F_1 S^1 \wedge T_+)^{\wedge n}$ has a filtration with quotients that are wedges of spectra, each of which is weakly equivalent to a spectrum of the form $\Sigma^{\infty} ((G \times_H (T^{\times n})^{\phi})_+ \wedge S^m)$ for some $m \geq 0$, $H \subseteq G$ and $\phi : H \to \Sigma_n$, and there is one of these for each $m$-cell of ${E_G \Sigma_n}$ of type $G \times \Sigma_n / H^{\phi}$. We can now prove the following. We use superscripts to denote skeleta.

\begin{lem}\label{lem:psisurj}
The map $\psi_T [n] (G/G)$ is surjective for each $n > 1$.
\end{lem}
\begin{proof}
Fix $n > 1$, and consider the component $\psi_T [n] (G/G)$. By the above, the map
\begin{align*}
	{E_G \Sigma_n}^{[0]}_+ \wedge_{\Sigma_n} (F_1 S^1 \wedge T_+)^{\wedge n} \to (F_1 S^1 \wedge T_+)^{\wedge n}/\Sigma_n
\end{align*}
is surjective on $\underline{\pi}_0$. If we take
\begin{align*}
	{E_G \Sigma_n}^{[0]} = \coprod_{H \subseteq G, \phi : H \to \Sigma_n} G \times \Sigma_n / H^{\phi},
\end{align*}
we obtain
\begin{align*}
	{E_G \Sigma_n}^{[0]}_+ \wedge_{\Sigma_n} (F_1 S^1 \wedge T_+)^{\wedge n} \cong \bigvee_{H, \phi} G_+ \wedge_H  ((F_1 S^1 \wedge T_+)^{\wedge n})^{\phi}.
\end{align*}
These wedge summands are equivalent to suspension spectra of finite $G$-sets; in fact we have
\begin{align*}
	G_+ \wedge_H  ((F_1 S^1 \wedge T_+)^{\wedge n})^{\phi} \xrightarrow{\sim} \Sigma^{\infty} (G \times_H (T^{\times n})^{\phi})_+.
\end{align*}
Note that $(T^{\times n})^{\phi}$ is the $H$-set exponential $T^{\{1,...,n\}}$, where $H$ acts on $\{1,...,n\}$ through $\phi$. Now if $H \neq G$, then every element of the associated Mackey functor evaluated at $G/G$ is a sum of elements that are transferred up from proper subgroups, and so maps to an element that is in the image of $\psi_T$ by our induction hypothesis. Thus it suffices to show that elements of the form
\begin{align*}
\xymatrix{
 & \ast \ar[dl]_-{=} \ar[dr]^-{f} & \\
 \ast & & (T^{\times n})^{\phi} }
\end{align*}
map into the image of $\psi_T [n] (G/G)$, where $\phi : G \to \Sigma_n$ is a homomorphism. Now the composite
\begin{align*}
	((F_1 S^1 \wedge T_+)^{\wedge n})^{\phi} &\subseteq {E_G \Sigma_n}^{[0]}_+ \wedge_{\Sigma_n} (F_1 S^1 \wedge T_+)^{\wedge n}\\
	&\to (F_1 S^1 \wedge T_+)^{\wedge n}/\Sigma_n
\end{align*}
is simply the quotient map when we use the isomorphisms given by Lemma~\ref{lem:freesigmaorbit}, so one sees that the image of the above span coincides with the image under $\psi_T [n] (G/G)$ of the diagram
\begin{align*}
\xymatrix{
 & \{1,...,n\} \ar[dl]_-{f'} \ar[r] & \ast \ar[dr]^-{=} & \\
 T & & & \ast }
\end{align*}
in $F_T^+ [n] (G/G)$, where $f'$ is the adjoint of $f$.
\end{proof}

The following algebraic short-circuit will simplify our remaining task.

\begin{lem}\label{lem:algtrick}
Let $f : A \to B$ and $g : B \to A$ be surjective maps of abelian groups. If either group is finitely generated, then both maps are isomorphisms.
\end{lem}

Recall from Section~\ref{sec:tambara} that each $F_T^+ [n] (G/G)$ is a finitely generated free abelian group. Thus, to show that $\psi_T [n] (G/G)$ is an isomorphism, we need only construct a surjective map
\begin{align*}
	\Xi : \CC_T [n] (G/G) \to F_T^+ [n] (G/G).
\end{align*}

There is no need to show that the maps are inverses. Now we have
\begin{align*}
	\CC_T [n] \cong \underline{\pi}_0 \Sigma^{\infty} ({E_G \Sigma_n}^{[1]} \times_{\Sigma_n} T^{\times n})_+,
\end{align*}

and we can take
\begin{align*}
	{E_G \Sigma_n}^{[1]} = \hocoeq \big( \coprod G \times \Sigma_n / L^{\lambda} \rightrightarrows \coprod_{H,\phi} G \times \Sigma_n / H^{\phi} \big),
\end{align*}

where the first coproduct is over all pairs of distinct maps from an orbit of this type to ${E_G \Sigma_n}^{[0]}$. We obtain an isomorphism as below,
\begin{align}\label{eq:colim}
	\CC_T [n] \cong \colim_{G \times \Sigma_n / H^{\phi}} [\spacedash, G \times_H (T^{\times n})^{\phi}]
\end{align}

where the colimit is over the full subcategory of orbits in $\Fin_{G \times \Sigma_n}$ with stabilizers in $\FF_G[n]$. To understand this colimit, we need the following lemmas.

\begin{lem}\label{lem:orbitmaps}
Let $H$ and $L$ be subgroups of $G$, and let $\phi : H \to \Sigma_n$ and $\lambda : L \to \Sigma_n$ be homomorphisms. There is a map in $\Fin_{G \times \Sigma_n}$
\begin{align*}
	G \times \Sigma_n / L^{\lambda} \to G \times \Sigma_n / H^{\phi}
\end{align*}
sending the identity coset to $(g,\sigma)H^{\phi}$ if and only if $L \subseteq gHg^{-1}$ and
\begin{align*}
	\lambda(l) = \sigma \phi(g^{-1} l g) \sigma^{-1}
\end{align*}
for all $l \in L$.
\end{lem}

\begin{lem}\label{lem:orbitmapsinduce}
Let $H$ and $L$ be subgroups of $G$, and let $\phi : H \to \Sigma_n$ and $\lambda : L \to \Sigma_n$ be homomorphisms. Suppose there is a map in $\Fin_{G \times \Sigma_n}$
\begin{align*}
	f : G \times \Sigma_n / L^{\lambda} \to G \times \Sigma_n / H^{\phi}
\end{align*}
sending the identity coset to $(g,\sigma)H^{\phi}$. Then the following diagram commutes, where the vertical isomorphisms are given by Lemma~\ref{lem:freesigmaorbit}.
\begin{align*}
\xymatrix{
 (G \times \Sigma_n / L^{\lambda}) \times_{\Sigma_n} T^{\times n} \ar[d]_-{\cong} \ar[rr]^-{f \times_{\Sigma_n} Id} & & (G \times \Sigma_n / H^{\phi}) \times_{\Sigma_n} T^{\times n} \ar[d]^-{\cong} \\
 G \times_{L} (T^{\times n})^{\lambda} \ar[rr]_-{[k,\{t_j\}_j] \mapsto [kg, \{ g^{-1}t_{\sigma(j)} \}_j]} & & G \times_{H} (T^{\times n})^{\phi} }
\end{align*}
\end{lem}

Recalling that $(T^{\times n})^{\phi}$ is the $H$-set exponential $T^{\{1,...,n\}}$, we have an evaluation $H$-map
\begin{align*}
	eval : (T^{\times n})^{\phi} \times \{ 1,...,n \} \to T.
\end{align*}

We denote the adjoint $G$-map
\begin{align*}
	G \times _H ((T^{\times n})^{\phi} \times \{ 1,...,n \}) \to T
\end{align*}

by $eval$ also. We also have a projection map
\begin{align*}
	proj : (T^{\times n})^{\phi} \times \{ 1,...,n \} \to (T^{\times n})^{\phi}.
\end{align*}

We now define maps
\begin{align*}
	\Xi_{H,\phi} : [\spacedash, G \times_{H} (T^{\times n})^{\phi}] \to F_T^+ [n]
\end{align*}

by the diagrams below.
\begin{align*}
\xymatrix{
 G \times_{H} ((T^{\times n})^{\phi} \times \{ 1,...,n \}) \ar[d]_-{eval} \ar[rr]^-{G \times_H proj} & & G \times_{H} (T^{\times n})^{\phi} \ar[d]^-{=} \\
 T & & G \times_{H} (T^{\times n})^{\phi} }
\end{align*}

Checking that the $\Xi_{H,\phi}$ induce a well-defined map of Mackey functors
\begin{align*}
	\Xi \defeq \colim \Xi_{H,\phi} : \CC_T [n] \to F_T^+ [n]
\end{align*}

under the identification~\ref{eq:colim} amounts to verifying the following lemma.

\begin{lem}\label{lem:mapcttoft}
Let $H$ and $L$ be subgroups of $G$, and let $\phi : H \to \Sigma_n$ and $\lambda : L \to \Sigma_n$ be homomorphisms. Suppose there is a map in $\Fin_{G \times \Sigma_n}$
\begin{align*}
	f : G \times \Sigma_n / L^{\lambda} \to G \times \Sigma_n / H^{\phi}
\end{align*}
sending the identity coset to $(g,\sigma)H^{\phi}$. Then there is a well-defined map as below.
\begin{align*}
	G \times_{L} ((T^{\times n})^{\lambda} \times \{ 1,...,n \}) &\to G \times_{H} ((T^{\times n})^{\phi} \times \{ 1,...,n \}) \\
	[k, \{ t_j \}_j, m] &\mapsto [kg, \{ g^{-1}t_{\sigma(j)} \}_j, \sigma^{-1}(m)]
\end{align*}
This map makes the diagram
\begin{align*}
\xymatrix{
 & G \times_{H} ((T^{\times n})^{\phi} \times \{ 1,...,n \}) \ar[dl]_-{eval} \ar[rr]^-{G \times_H proj} && G \times_{H} (T^{\times n})^{\phi} \\
 T & G \times_{L} ((T^{\times n})^{\lambda} \times \{ 1,...,n \}) \ar[l]^-{eval} \ar[u] \ar[rr]_-{G \times_{L} proj} && G \times_{L} (T^{\times n})^{\lambda} \ar[u] }
\end{align*}
commute and the square a pullback, where the right vertical map is given by Lemma~\ref{lem:orbitmapsinduce}.
\end{lem}

It remains to show that the resulting map $\Xi (G/G)$ is surjective. It suffices to show that any diagram
\begin{align*}
\xymatrix{
 & U \ar[dl] \ar[r] & V \ar[dr] & \\
 T & & & \ast }
\end{align*}

with $V$ an orbit is in the image. We may assume $V = G/H$ for some subgroup $H$ of $G$. Then $U \cong G \times_H U'$ for some $H$-set $U'$ with $n$ elements. Now choose a bijection
\begin{align*}
	\{ 1, ... , n \} \xrightarrow{\cong} U'
\end{align*}

and let $\phi : H \to \Sigma_n$ be the induced action map. We now see that our diagram is isomorphic to one of the form below.
\begin{align*}
\xymatrix{
 & G \times_H \{ 1, ... , n \} \ar[dl]_-{f} \ar[rr]^-{G \times_H proj} && G/H \ar[dr] & \\
 T & & & & \ast }
\end{align*}

Now $f$ is adjoint to an $H$-map $\{ 1, ... , n \} \to T$, which in turn is adjoint to an $H$-map
\begin{align*}
	f' : \ast \to (T^{\times n})^{\phi}.
\end{align*}

One can now check that our diagram is the image under $\Xi$ of the span
\begin{align*}
\xymatrix{
 & G/H \ar[dl] \ar[dr]^-{G \times_H f'} & \\
 \ast & & G \times_H (T^{\times n})^{\phi} }
\end{align*}

by verifying the following lemma.

\begin{lem}\label{lem:mapbacksurj}
The following diagram commutes, and the square is a pullback.
\begin{align*}
\xymatrix{
 & G \times_H ((T^{\times n})^{\phi} \times \{ 1, ... , n \}) \ar[dl]_-{eval} \ar[rr]^-{G \times_H proj} && G \times_H (T^{\times n})^{\phi} \\
 T & G \times_H \{ 1, ... , n \} \ar[l]^-{f} \ar[u]_-{G \times_H (f' \times Id)} \ar[rr]_-{G \times_H proj} && G/H \ar[u]_-{G \times_H f'} }
\end{align*}
\end{lem}

We have completed the proof of Theorem~\ref{thm:gentamb}.

\indent \emph{Remark:} It should not be difficult to verify that $\Xi$ is actually the inverse of $\psi_T [n]$. However, in view of Lemma~\ref{lem:algtrick}, it would not matter even if this turned out not to be true.

\section{The Main Theorems}\label{sec:main}

We can now quickly derive our main theorems.

\begin{thm}\label{thm:getanything}
Every Tambara functor occurs as $\underline{\pi}_0$ of a commutative ring spectrum. In fact, there even exists a functor
\begin{align*}
	EM : Tamb(G) \to comm_G
\end{align*}
taking cofibrant, fibrant and Eilenberg MacLane values such that the composite $\underline{\pi}_0 \circ EM$ is naturally isomorphic to the identity.
\end{thm}
\begin{proof}
Let $\underline{R}$ be a Tambara functor. For each subgroup $H$ of $G$ and each $x \in \underline{R} (G/H)$ we have a corresponding map $x : [\spacedash, G/H] \to \underline{R}$ of Mackey functors, and hence an induced map
\begin{align*}
	F_{G/H}^+ \xrightarrow{x_*} \underline{R}
\end{align*}
of Tambara functors. Define a functor
\begin{align*}
	EM_0 : Tamb(G) \to comm_G
\end{align*}
by
\begin{align*}
	EM_0 (\underline{R}) \defeq \coprod_{H,x} \CC (F_1 S^1 \wedge G/H_+),
\end{align*}
so that we have a natural isomorphism
\begin{align*}
	\underline{\pi}_0 EM_0 (\underline{R}) \cong \coprod_{H,x} F_{G/H}^+
\end{align*}
by Proposition~\ref{prop:commcoprod} and Theorem~\ref{thm:gentamb}, and hence a natural surjection
\begin{align*}
	\coprod_{H,x} x_* : \underline{\pi}_0 EM_0 (\underline{R}) \to \underline{R}.
\end{align*}
Next, apply a functorial fibrant replacement to $EM_0$ to obtain $EM_1$; we may identify $\underline{\pi}_0 EM_1$ with $\underline{\pi}_0 EM_0$ (canonically). Now, for each subgroup $K$ of $G$ and $y$ in the kernel of the above map at $G/K$ there is a corresponding map $y : [\spacedash, G/K] \to \underline{\pi}_0 EM_1 (\underline{R})$; choose a representing map of spectra
\begin{align*}
	y : F_1 S^1 \wedge G/K_+ \to EM_1 (\underline{R})
\end{align*}
and let the induced map in $comm_G$ be
\begin{align*}
	y_* : \CC (F_1 S^1 \wedge G/K_+) \to EM_1 (\underline{R}).
\end{align*}
Now define $EM_2 (\underline{R})$ by the pushout diagram below, where we give the unit interval $I = [0,1]$ the basepoint $1$.
\begin{align*}
\xymatrix{
 \coprod_{K,y} \CC (F_1 S^1 \wedge G/K_+) \ar[d] \ar[r]^-{\coprod_{K,y} y_*} & EM_1 (\underline{R}) \ar[d] \\
 \coprod_{K,y} \CC (F_1 S^1 \wedge G/K_+ \wedge I) \ar[r] & EM_2 (\underline{R}) }
\end{align*}
By Propositions~\ref{prop:commcoprod} and~\ref{prop:commpushout} and Theorem~\ref{thm:gentamb} we see that the map $\underline{\pi}_0 EM_1 (\underline{R}) \to \underline{R}$ descends to an isomorphism
\begin{align*}
	\underline{\pi}_0 EM_2 (\underline{R}) \xrightarrow{\cong} \underline{R}.
\end{align*}
We can make this construction functorial by using \emph{all maps} $y_*$ instead of just one from each homotopy class. We now apply Proposition~\ref{prop:commpost0} to the output of $EM_2$ to obtain our functor $EM$.
\end{proof}

\begin{thm}\label{thm:homhocomm}
If $X \in comm_G$ is $(-1)$-connected and $H\underline{R} \in comm_G$ is Eilenberg MacLane then the functor $\underline{\pi}_0$ induces an isomorphism
\begin{align*}
	(\underline{\pi}_0)_* : Hom_{Ho(comm_G)} (X, H\underline{R}) \xrightarrow{\cong} Hom_{Tamb(G)} (\underline{\pi}_0 X, \underline{R}).
\end{align*}
\end{thm}
\begin{proof}
We may assume that $X$ and $H\underline{R}$ are both fibrant. Using the notation from the proof of Theorem~\ref{thm:getanything}, it is a simple matter to construct a map
\begin{align*}
	EM_2 (\underline{\pi}_0 X) \to X
\end{align*}
which induces an isomorphism on $\underline{\pi}_0$. Then by Proposition~\ref{prop:commpost0}, it induces an isomorphism on $Hom_{Ho(comm_G)} (\spacedash, H\underline{R})$ as well. Hence we may replace $X$ by $EM_2 (\underline{\pi}_0 X)$. The result now follows from the explicit construction of $EM_2$ and standard adjunctions.
\end{proof}

Combining the above two theorems, we have the following.

\begin{thm}\label{thm:hoemcomm}
The functor $\underline{\pi}_0$ induces an equivalence of categories from the homotopy category of Eilenberg MacLane commutative ring spectra to the category of $G$-Tambara functors.
\end{thm}

Next we obtain some algebraic corollaries. First, combining Proposition~\ref{prop:commcoprod} and Theorem~\ref{thm:getanything} we get coproducts in $Tamb(G)$.

\begin{cor}\label{cor:tambcoprod}
Let $\{ X_{\alpha} \}$ be a collection of Tambara functors. The coproduct of the $X_{\alpha}$'s in the category of commutative Green functors has a unique Tambara functor structure inducing the canonical product such that the canonical inclusions are maps of Tambara functors, and this structure makes it the coproduct in $Tamb(G)$.
\end{cor}

\indent \emph{Remark:} The above corollary is proven algebraically in~\cite{Stri}.\\
\indent We now obtain pushouts by applying Proposition~\ref{prop:commpushout} and Theorem~\ref{thm:getanything}.

\begin{cor}\label{cor:tambpushout}
If $\underline{R}_0 \to \underline{R}_1$ and $\underline{R}_0 \to \underline{R}_2$ are maps of Tambara functors, then the pushout in the category of commutative Green functors has a unique Tambara functor structure inducing the canonical product such that $\underline{R}_1 \to \underline{R}_1 \otimes_{\underline{R}_0} \underline{R}_2$ and $\underline{R}_2 \to \underline{R}_1 \otimes_{\underline{R}_0} \underline{R}_2$ are maps of Tambara functors, and this structure makes it the pushout in $Tamb(G)$.
\end{cor}

Since all colimits can be obtained from coproducts and pushouts, we obtain the following.

\begin{cor}\label{cor:tambcocomplete}
The category of Tambara functors is cocomplete, and the forgetful functor to the category of commutative Green functors preserves colimits.
\end{cor}

Next, since restriction functors are exact and symmetric monoidal, Corollaries~\ref{cor:tambcoprod} and~\ref{cor:tambpushout} imply the following.

\begin{cor}\label{cor:rescolimit}
If $H$ is a subgroup of $G$ then the restriction functor
\begin{align*}
	\Res_H^G : Tamb(G) \to Tamb(H)
\end{align*}
preserves colimits.
\end{cor}

\indent \emph{Remark:} There is a much simpler proof of this corollary: the induction functor actually induces a right adjoint
\begin{align*}
	\Ind_H^G : Tamb(H) \to Tamb(G)
\end{align*}
to $\Res_H^G$.\\
\indent Next, we can prove the existence of free Tambara functors using the fact that the free commutative ring spectrum on a $(-1)$-connected spectrum is $(-1)$-connected and Theorem~\ref{thm:homhocomm}. In the following we use $\CC$ to denote the \emph{derived} free commutative ring spectrum functor.

\begin{cor}\label{cor:existfreetambara}
If $\underline{M}$ is a Mackey functor then $\underline{\pi}_0 \CC (H\underline{M})$ is a free Tambara functor on $\underline{M}$. Hence the forgetful functor from $Tamb(G)$ to $Mack(G)$ has a left adjoint as below.
\begin{align*}
	\TT : Mack(G) &\to Tamb(G) \\
	\underline{M} &\to \underline{\pi}_0 \CC (H\underline{M})
\end{align*}
\end{cor}
\begin{proof}
Let $\underline{M}$ be a Mackey functor and $\underline{R}$ a Tambara functor. We have
\begin{align*}
	Hom_{Mack(G)} (\underline{M}, \underline{R}) &\cong Hom_{Ho(Sp_G)} (H\underline{M}, H\underline{R}) \\
	&\cong Hom_{Ho(comm_G)} (\CC (H\underline{M}), H\underline{R}) \\
	&\cong Hom_{Tamb(G)} (\underline{\pi}_0 \CC (H\underline{M}), \underline{R}).
\end{align*}
\end{proof}

\indent The following is immediate from the explicit form of free Tambara functors. Again, there is a much simpler proof: $\TT$ and $\Res_H^G$ are left adjoints whose right adjoints (the forgetful functor and $\Ind_H^G$, resp.) commute.

\begin{cor}\label{cor:resfreetambara}
Let $H$ be a subgroup of $G$. If $f : \underline{M} \to \underline{R}$ expresses $\underline{R}$ as a free $G$-Tambara functor on $\underline{M}$ then its restriction
\begin{align*}
	\Res_H^G f : \Res_H^G \underline{M} \to \Res_H^G \underline{R}
\end{align*}
expresses $\Res_H^G \underline{R}$ as a free $H$-Tambara functor on $\Res_H^G \underline{M}$.
\end{cor}
\indent Finally, let $H$ be a subgroup of $G$, and recall that the norm functor
\begin{align*}
	N_H^G : comm_H \to comm_G
\end{align*}

of~\cite{HHR} gives the left adjoint of the restriction functor. Note that, if $X \in comm_H$ is cofibrant and $(-1)$-connected, then so is $N_H^G X$. This is because $X$ can be given a cell structure with no negative cells; applying $N_H^G$, we obtain a cell structure for $N_H^G X$ with no negative cells.

\begin{cor}\label{cor:normleftadjoint}
Let $H$ be a subgroup of $G$. The restriction functor $\Res_H^G : Tamb(G) \to Tamb(H)$ has a left adjoint, given below.
\begin{align*}
	\underline{R} \mapsto \underline{\pi}_0 N_H^G (H\underline{R})
\end{align*}
\end{cor}
\begin{proof}
Let $\underline{R}_1 \in Tamb(H)$ and $\underline{R}_2 \in Tamb(G)$. We have
\begin{align*}
	Hom_{Tamb(H)} (\underline{R}_1, \Res_H^G \underline{R}_2) &\cong Hom_{Ho(comm_H)} (H\underline{R}_1, \Res_H^G H\underline{R}_2) \\
	&\cong Hom_{Ho(comm_G)} (N_H^G H\underline{R}_1, H\underline{R}_2) \\
	&\cong Hom_{Tamb(G)} (\underline{\pi}_0 N_H^G (H\underline{R}_1), \underline{R}_2).
\end{align*}
\end{proof}

We wish to identify the above left adjoint with the derived norm functor shown below, where $H\underline{M}$ is taken to be cofibrant.
\begin{align*}
	N_H^G : Mack(H) &\to Mack(G) \\
	\underline{M} &\mapsto \underline{\pi}_0 N_H^G (H\underline{M})
\end{align*}

For this we require two more lemmas.

\begin{lem}\label{lem:normpi0}
If $f : X \to Y$ is a map of $(-1)$-connected, cofibrant $H$-spectra inducing an isomorphism on $\underline{\pi}_0$, then
\begin{align*}
	N_H^G f : N_H^G X \to N_H^G Y
\end{align*}
induces an isomorphism on $\underline{\pi}_0$.
\end{lem}

This may be proven by applying Theorem~I.5.9 of~\cite{Ull} to the natural cofiber sequences $Post_1 Z \to Z \to Post^0 Z$. It implies that the derived norm functor on Mackey functors is symmetric monoidal, and thus defines a functor from $Comm(H)$ to $Comm(G)$.

\begin{lem}\label{lem:veryflat}
If $X \in comm_H$ is cofibrant then $N_H^G X$ is equivalent to the derived norm in $Ho(Sp_G)$.
\end{lem}

This lemma (for the classical model structure) is Proposition~B.63 of~\cite{HHR}.

\begin{cor}\label{cor:leftadjointderivednorm}
Let $H$ be a subgroup of $G$. The left adjoint of the restriction functor $\Res_H^G : Tamb(G) \to Tamb(H)$ coincides with the derived norm functor
\begin{align*}
	N_H^G : Mack(H) \to Mack(G)
\end{align*}
on underlying commutative Green functors.
\end{cor}

In a subsequent paper we shall give detailed algebraic descriptions of the effects of $\CC$ and $N_H^G$ on Mackey functors. We shall also give an algebraic demonstration of the adjunction in Corollary~\ref{cor:leftadjointderivednorm}.

\section{Tambara Functors vs Commutative Green Functors}\label{sec:cgreen}

In this section we investigate the extent to which $Tamb(G)$ differs from $Comm(G)$, and prove that there can be no lax symmetric monoidal construction of Eilenberg MacLane spectra. This section is independent of the previous sections.

\begin{prop}\label{prop:greennottamb}
If $G$ is nontrivial then there exists a commutative Green functor which does not arise from a Tambara functor.
\end{prop}
\begin{proof}
Suppose the statement is false. Let $p$ be a prime divisor of $|G|$ and let $C_p$ be a cyclic subgroup of order $p$. Define a commutative Green functor $\underline{R}$ over $C_p$ by
\begin{align*}
	\underline{R}(C_p/e) &= \ZZ/p\ZZ[x] \\
	\underline{R}(C_p/C_p) &= \ZZ/p\ZZ
\end{align*}
with obvious ring structures; we define the conjugations to be the identity on $\underline{R}(C_p/e)$, the restriction map $\underline{R}(C_p/C_p) \to \underline{R}(C_p/e)$ to be the obvious inclusion, and the transfer $\underline{R}(C_p/e) \to \underline{R}(C_p/C_p)$ to be zero. Now, since the restriction functor
\begin{align*}
	\Res_{C_p}^G : Mack(G) \to Mack(C_p)
\end{align*}
is symmetric monoidal, its right adjoint
\begin{align*}
	\Ind_{C_p}^G : Mack(C_p) \to Mack(G)
\end{align*}
is lax symmetric monoidal. Hence, $\Ind_{C_p}^G \underline{R}$ is a commutative Green functor. Then by our assumption, it can be given norm maps to make it a Tambara functor. Thus we may assume that $\Res_{C_p}^G \Ind_{C_p}^G \underline{R}$ is a Tambara functor and that the counit map
\begin{align*}
	\epsilon : \Res_{C_p}^G \Ind_{C_p}^G \underline{R} \to \underline{R}
\end{align*}
is a map of commutative Green functors. Recall that
\begin{align*}
	\Res_{C_p}^G \Ind_{C_p}^G \underline{R} (T) = \underline{R} (\Res_{C_p}^G (G \times_{C_p} T)),
\end{align*}
and that the map $\epsilon$ is induced by the inclusion
\begin{align*}
	T \to \Res_{C_p}^G (G \times_{C_p} T) \cong T \textstyle \coprod \big(\coprod_{G/C_p-C_p} T \big)
\end{align*}
corresponding to the identity coset; hence, $\epsilon$ is surjective. Choosing an element $\bar{x} \in \Res_{C_p}^G \Ind_{C_p}^G \underline{R} (C_p/e)$ such that $\epsilon (\bar{x}) = x$, we have
\begin{align*}
	r_e^{C_p} \epsilon (n_e^{C_p} \bar{x}) &= \epsilon (r_e^{C_p} n_e^{C_p} \bar{x}) \\
	&= \epsilon \big( \textstyle \prod_{g \in C_p} g \cdot \bar{x} \big) \\
	&= \textstyle \prod_{g \in C_p} g \cdot \epsilon(\bar{x}) \\
	&= x^p,
\end{align*}
but $x^p$ is not in the image of $r_e^{C_p}$, so we have a contradiction.
\end{proof}

\indent \emph{Remark:} Other examples of commutative Green functors which are not Tambara functors are given in~\cite{Mazur}.\\
\indent The following is immediate.

\begin{thm}\label{thm:nosymmonem}
If $G$ is nontrivial, then there is no lax symmetric monoidal construction of Eilenberg MacLane $G$-spectra.
\end{thm}
\begin{proof}
If there were a lax symmetric monoidal functor from $Mack(G)$ to any "good" category of $G$-spectra, then it would send commutative monoids (commutative Green functors) to commutative ring spectra. Applying the functor $\underline{\pi}_0$, this would imply that all commutative Green functors arise from Tambara functors, contradicting Proposition~\ref{prop:greennottamb}.
\end{proof}

\begin{prop}\label{prop:manytambaratogreen}
If $G$ is nontrivial then there exist commutative Green functors with arbitrarily many distinct Tambara functor structures.
\end{prop}
\begin{proof}
Let $p$ be a prime divisor of $|G|$ and let $C_p$ be a cyclic subgroup of order $p$. Let $S$ be a set of arbitrary cardinality. Define a commutative Green functor $\underline{R}$ over $C_p$ by
\begin{align*}
	\underline{R}(C_p/e) &= \ZZ/p\ZZ[x] \\
	\underline{R}(C_p/C_p) &= \ZZ/p\ZZ[x_s : s \in S]
\end{align*}
with obvious polynomial ring structures; we define the conjugations to be the identity on $\underline{R}(C_p/e)$, the restriction map $\underline{R}(C_p/C_p) \to \underline{R}(C_p/e)$ to be the map sending all $x_s$ to $x$, and the transfer $\underline{R}(C_p/e) \to \underline{R}(C_p/C_p)$ to be zero. Then Tambara functor structures on $\underline{R}$ correspond to ring homomorphisms
\begin{align*}
	n_e^{C_p} : \underline{R}(C_p/e) \to \underline{R}(C_p/C_p)
\end{align*}
such that the composite $r_e^{C_p} n_e^{C_p}$ is the Frobenius $x \mapsto x^p$. Letting $m$ denote a monomial in the $x_s$'s of degree $p$, we denote by $\underline{R}_m$ the Tambara functor such that $n_e^{C_p} (x) = m$. We claim that the Tambara functors $\Ind_{C_p}^G \underline{R}_m$ are all distinct. It suffices to show that their restrictions $\Res_{C_p}^G \Ind_{C_p}^G \underline{R}_m$ are distinct. Recall from the proof of Proposition~\ref{prop:greennottamb} that the counit map
\begin{align*}
	\epsilon_m : \Res_{C_p}^G \Ind_{C_p}^G \underline{R}_m \to \underline{R}_m
\end{align*}
is surjective. It is also a map of Tambara functors for each $m$. Let $\bar{x}$ be an element of $\Res_{C_p}^G \Ind_{C_p}^G \underline{R} (G/e)$ such that $\epsilon (\bar{x}) = x$. Then regarding $\bar{x}$ as an element in $\Res_{C_p}^G \Ind_{C_p}^G \underline{R}_m (C_p/e)$ we have
\begin{align*}
	\epsilon_m (n_e^{C_p} \bar{x}) = n_e^{C_p} (\epsilon_m \bar{x}) = n_e^{C_p} x = m,
\end{align*}
so we get a different value of $n_e^{C_p} \bar{x}$ for each $m$.
\end{proof}

Next, let $F$ be the forgetful functor from $Tamb(G)$ to $Comm(G)$, and let $Im(F)$ denote the full subcategory of objects in the image of $F$. We have the following negative result.

\begin{prop}\label{prop:nosplitting}
If $G$ is nontrivial then there does not exist a functor
\begin{align*}
	s : Im(F) \to Tamb(G)
\end{align*}
such that the composite $F \circ s$ is naturally isomorphic to the identity.
\end{prop}
\begin{proof}
Let $p$ be a prime divisor of $|G|$ and let $C_p$ be a cyclic subgroup of order $p$. Let $S$ be an infinite set. Let $\underline{R}$ be as in the proof of Proposition~\ref{prop:manytambaratogreen}. Now define a Tambara functor $\underline{B}$ over $C_p$ by
\begin{align*}
	\underline{B}(C_p/e) &= \ZZ/p\ZZ[x] \\
	\underline{B}(C_p/C_p) &= \ZZ/p\ZZ[x^p]
\end{align*}
with obvious polynomial ring structures; we define the conjugations to be the identity on $\underline{B}(C_p/e)$, the restriction map $\underline{B}(C_p/C_p) \to \underline{B}(C_p/e)$ to be the map sending $x^p$ to $x^p$, the transfer $\underline{B}(C_p/e) \to \underline{B}(C_p/C_p)$ to be zero and the norm $n_e^{C_p}$ to be the map sending $x$ to $x^p$. Note that this is the \emph{unique} Tambara functor structure on $\underline{B}$ with the same product. Next, for any monomial $m$ in the $x_s$'s of degree $p$, define a map of commutative Green functors $f_m : \underline{B} \to \underline{R}$ by $f_m (x) = x$ in the $C_p/e$ component and $f_m (x^p) = m$ in the $C_p/C_p$ component. Now, suppose that there exists a splitting functor $s$; we may assume that $F \circ s$ is precisely the identity functor. Since the restrictions of $\underline{B}$ are injective, the same is true of $\Ind_{C_p}^G \underline{B}$; hence, the commutative Green functor $\Ind_{C_p}^G \underline{B}$ also has a unique Tambara functor structure. Now we have natural isomorphisms as below,
\begin{align*}
	\Ind_{C_p}^G \underline{M} (G/e) &\cong \bigtimes_{G/C_p} \underline{M} (C_p/e) \\
	\Ind_{C_p}^G \underline{M} (G/G) &\cong \underline{M} (C_p/C_p)
\end{align*}
so by applying our splitting functor $s$, we obtain commutative diagrams
\begin{align*}
\xymatrix{
 \bigtimes_{G/C_p} \underline{B} (C_p/e) \ar[d]_-{\bigtimes_{G/C_p} f_m} \ar[rr]^-{\bigtimes_{C/C_p} n_e^{C_p}} & & \underline{B} (C_p/C_p) \ar[d]^-{f_m} \\
 \bigtimes_{G/C_p} \underline{R} (C_p/e) \ar[rr]_-{n_e^G} & & \underline{R} (C_p/C_p) }
\end{align*}
for all $m$ and a fixed Tambara functor structure on $\Ind_{C_p}^G \underline{R}$. Since the top horizontal map is surjective, this implies that $f_m$ is determined by the $C_p/e$ component; a contradiction.
\end{proof}

The following is immediate.

\begin{thm}\label{thm:nosplittingspectra}
If $G$ is nontrivial then there does not exist a functor
\begin{align*}
	s : Im(F) \to Ho(comm_G)
\end{align*}
such that the composite $\underline{\pi}_0 \circ s$ is naturally isomorphic to the identity.
\end{thm}

Thus, not all commutative Green functors arise from Tambara functors, those that do may arise from arbitrarily many distinct Tambara functors, and there is not even a natural way of choosing norm maps for these.

\vspace{2cm}

\bibliographystyle{alphanum}
\bibliography{tambarabiblio}

\newcommand{\noopsort}[1]{} \newcommand{\printfirst}[2]{#1}
  \newcommand{\singleletter}[1]{#1} \newcommand{\switchargs}[2]{#2#1}
\begin{thebibliography}{MMSS}

\bibitem[Bru]{Brun}
M.~Brun.
\newblock Witt vectors and equivariant ring spectra applied to cobordism.
\newblock {\em Proc. Lond. Math. Soc. (3)}, 94(2):351--385, 2007.

\bibitem[HHR]{HHR}
Michael~A. Hill, Michael~J. Hopkins, and Douglas~C. Ravenel.
\newblock On the non-existence of elements of {K}ervaire invariant one, 2009.
\newblock arXiv:0908.3724.

\bibitem[Maz]{Mazur}
Kristen Mazur.
\newblock {\em On the Structure of Mackey Functors and Tambara Functors}.
\newblock PhD thesis, University of Virginia, 2013.

\bibitem[MM]{MM}
M.~A. Mandell and J.~P. May.
\newblock Equivariant orthogonal spectra and {$S$}-modules.
\newblock {\em Mem. Amer. Math. Soc.}, 159(755), 2002.

\bibitem[MMSS]{MMSS}
M.~A. Mandell, J.~P. May, S.~Schwede, and B.~Shipley.
\newblock Model categories of diagram spectra.
\newblock {\em Proc. London Math. Soc. (3)}, 82(2):441--512, 2001.

\bibitem[Sto]{Stolz}
Martin Stolz.
\newblock {\em Equivariant structure on smash powers of commutative ring
  spectra}.
\newblock PhD thesis, University of Bergen, 2011.

\bibitem[Str]{Stri}
Neil Strickland.
\newblock Tambara functors, 2012.
\newblock arXiv:1205.2516.

\bibitem[Tam]{Tambara}
D.~Tambara.
\newblock On multiplicative transfer.
\newblock {\em Comm. Algebra}, 21(4):1393--1420, 1993.

\bibitem[Ull]{Ull}
John Ullman.
\newblock {\em On the regular slice spectral sequence}.
\newblock PhD thesis, MIT, 2013.
\newblock available at http://math.mit.edu/$\sim$jrullman/thesis.pdf.

\end{thebibliography}

\end{document}